\documentclass[12pt,a4paper]{amsart}
\usepackage{graphics}
\usepackage{epsfig}
\usepackage{graphicx}
\usepackage[all]{xy}
\theoremstyle{plain}
\usepackage{amssymb}

\advance\hoffset-20mm \advance\textwidth40mm

\newtheorem{theorem}{Theorem}
\newtheorem{lemma}{Lemma}
\newtheorem{proposition}{Proposition}
\newtheorem{corollary}{Corollary}
\theoremstyle{definition}

\newtheorem*{definition*}{Definition}
\newtheorem{example}{Example}
\newtheorem{remark}{Remark}

\def\modd{{\rm mod}}

\def\KK{{\mathbb K}}

\def\GG{{\mathbb G}}

\def\PP{{\mathbb P}}
\def\FF{{\mathbb F}}
\def\gg{\mathfrak{g}}
\def\gl{\mathfrak{gl}}
\def\mm{\mathfrak{m}}

\def\GL{{\rm GL}}
\def\Ann{{\rm Ann}}
\def\Mat{{\rm Mat}}

\def\SL{{\rm SL}}
\def\SO{{\rm SO}}

\def\Ker{{\rm Ker}\,}

\begin{document}
\sloppy

\title[Hassett-Tschinkel correspondence]
{Hassett-Tschinkel correspondence: Modality and projective hypersurfaces}
\author[Ivan Arzhantsev]{Ivan V. Arzhantsev}
\thanks{Supported by RFBR grants 09-01-99648- , 09-01-90416-Ukr-f-a, and
the Deligne 2004 Balzan Prize in Mathematics.}
\address{Department of Higher Algebra, Faculty of Mechanics and Mathematics,
Moscow State University, Leninskie Gory 1, GSP-1, Moscow, 119991, Russia }
\email{arjantse@mccme.ru}
\author[Elena Sharoyko]{Elena V. Sharoyko}
\address{Department of Higher Algebra, Faculty of Mechanics and Mathematics,
Moscow State University, Leninskie Gory 1, GSP-1, Moscow, 119991, Russia }
\email{sharojko@mccme.ru}
\begin{abstract}
In~1999, B.~Hassett and Yu.~Tschinkel introduced a remarkable correspondence between generically transitive actions of the commutative unipotent algebraic group $\GG_a^n$ and finite-dimensional local algebras. In this paper we develop Hassett-Tschinkel correspondence and calculate
modality of generically transitive $\GG_a^n$-actions on projective spaces, classify actions of modality one, and characterize generically transitive $\GG_a^n$-actions on projective hypersurfaces of given degree. In particular, actions on degenerate projective quadrics are studied.
\end{abstract}
%
\maketitle
%
%
\section*{Introduction}
The theory of toric varieties plays an important role in modern Geometry, Algebra, Topology and Combinatorics.
It is caused by a beautiful description of toric varieties in terms of convex geometry.
Actual efforts to generalize such a description to other classes of objects are very natural. Let us recall that an  irreducible normal algebraic variety $X$ is called toric if there is a regular action of an algebraic torus $T$ on $X$ with an open orbit, see, for example, \cite{Fu}. This definition can be generalized in different ways. One possibility is to
consider arbitrary torus actions on algebraic varieties. Recently a semi-combinatorial description of such actions in
terms of so-called polyhedral divisors living on varieties of smaller dimension was introduced in \cite{AH} and \cite{AHS}.
Another variant is to restrict the (complex) algebraic torus action on a toric variety to the maximal compact
subtorus $(S^1)^n$, to axiomatize this class of $(S^1)^n$-actions, and to consider such actions on wider classes of topological spaces.
This is an active research area called Toric Topology, see~\cite{BP}. Further, one may consider algebraic actions with
an open orbit replacing the torus $T$ with a non-abelian connected reductive algebraic group $G$. The study of generically transitive $G$-actions leads to the theory of equivariant embeddings of homogeneous spaces $G/H$, where $H$ is an algebraic subgroup of $G$. This theory is well-developed. In particular,
there are classes of homogeneous spaces with combinatorial description of embeddings. These are the so-called spherical homogeneous spaces or, more generally, homogeneous spaces of complexity not exceeding one.
The description of embeddings here is much more complicated than in the toric case, see~\cite{LV}, \cite{Ti}.

It is also natural to replace the torus $T$ with the commutative unipotent group  $\GG_a^n=\GG_a\times\dots\times\GG_a$ ($n$ times), where $\GG_a$ is the additive group of the ground field $\KK$. The theory of generically transitive $\GG_a^n$-actions may be regarded as an "additive analogue" of toric geometry. But in this case we come across principal differences. It is easy to see that a generically transitive action of a unipotent group on an affine variety is transitive \cite[Section~1.3]{VP}.
So, in contrast with the toric case, we can not cover a variety with generically transitive $\GG_a^n$-action by invariant affine
charts. Also it is known that any toric variety contains finitely many $T$-orbits and that an isomorphism between toric varieties as algebraic varieties provides their isomorphism in the category of toric varieties \cite[Theorem~4.1]{Be}. In additive case these properties do not hold: one may consider two $\GG_a^2$-actions on the projective plane $\PP^2$ given in homogeneous coordinates as
$$
[z_0:z_1:z_2]\to [z_0:z_1+a_1z_0:z_2+a_2z_0]
$$
and
$$
[z_0:z_1:z_2]\to [z_0:z_1+a_1z_0:z_2+a_1z_1+(\frac{a_1^2}{2}+a_2)z_0].
$$
In the first case, there is a line consisting of fixed points, while for the second action there are three $\GG_a^2$-orbits.
Finishing with these negative observations, we turn to a positive side.

In~\cite{HT}, an algebraic interpretation of generically transitive $\GG_a^n$-actions on $\PP^n$ is given. Namely, it is shown there that such actions correspond to local finite-dimensional algebras. The aim of the present paper is to study this correspondence in more details. In Section~1 we recall basic facts on generically transitive $\GG_a^n$-actions. There is an initial correspondence between cyclic rational faithful $(m+1)$-dimensional $\GG_a^n$-modules $V$ and isomorphism classes of pairs $(R,U)$, where $R$ is a local associative commutative $\KK$-algebra with an $m$-dimensional maximal ideal $\mm$ and $U$ is an $n$-dimensional subspace of $\mm$ that generates the algebra $R$. These data determine a generically transitive $\GG_a^n$-action on the closure of the orbit $\GG_a^n\langle v\rangle$ in the projectivization $\PP^m$ of the space $V$
with $v\in V$ being a cyclic vector. Conversely, every generically transitive $\GG_a^n$-action on a normal projective variety arises in this way.
Suppose that the subspace $U$ coincides with the maximal ideal $\mm$, or equivalently, $n=m$. Then we obtain a generically transitive $\GG_a^n$-action on the projective space $\PP^n$.

Let $G$ be an algebraic group acting on an algebraic variety $X$.
Recall that {\it modality} of the action is the maximal number of parameters in a continuous family of $G$-orbits on $X$. In Section~\ref{sec2} modality of a generically transitive $\GG_a^n$-action on $\PP^n$ is calculated in terms of the corresponding local algebra $R$. Also some estimates on modality are given. The next section is devoted to classification of $\GG_a^n$-actions on $\PP^n$ of modality one. The classification results in series of finite-dimensional 2-generated local algebras (Theorem~\ref{Tclas}). It implies that for every given $n$ the number of actions of modality one is finite and for $n\ge 5$ there are exactly $n+1$ generically transitive $\GG_a^n$-actions on $\PP^n$ of modality one. On the contrary, actions of modality two admit moduli spaces of positive dimension. It brings some analogy with the embedding theory for homogeneous spaces of reductive groups with small complexity, cf.~\cite{Ti}.

Starting from Section~\ref{sec3} we study generically transitive $\GG_a^n$-actions on projective hypersurfaces.
These actions correspond to the case when $U$ is a hyperplane in the ideal $\mm$. Generically transitive $\GG_a^n$-actions on the non-degenerate quadric $Q_n\subset\PP^{n+1}$ are described in~\cite{Sha}. It is shown there that for every $n$ such an action is unique up to isomorphism, and the corresponding pair $(R,U)$ is indicated. We prove in Section~\ref{sec3} that families (of isomorphism classes) of
generically transitive $\GG_a^n$-actions on a degenerate quadric admit moduli.
In Theorem~\ref{qth} we characterize pairs $(R,U)$ representing actions on quadrics and develop an algebraic techniques to deal with such actions.
More generally, Theorem~\ref{dth} states that if $U$ is a hyperplane in $\mm$ and the pair $(R,U)$ determines a generically transitive $\GG_a^n$-action on a hypersurface $X\subset\PP^{n+1}$, then the degree of $X$ equals the maximal exponent $d$ such that $\mm^d$ is not contained in $U$.

In the last section we discuss possible generalizations, namely, $\GG_a^n$-structures on flag varieties and toric varieties.

\smallskip

We use \cite{Ha} as a general reference for algebraic geometry, and \cite{Hu} for algebraic groups and their representations.
The ground field $\KK$ is assumed to be algebraically closed and of characteristic zero.

%
\section{Hassett-Tschinkel correspondence}
\label{sec1}

In this section we recall some results from~\cite{HT}, see also \cite[\S~2]{Sha}.
Let $\rho:\GG_a^n\to\GL_{m+1}(\KK)$ be a faithful rational representation.
The differential defines a representation $d\rho:\gg\to\gl_{m+1}(\KK)$ of the tangent algebra $\gg = \text{Lie}(\GG_a^n)$ and the induced representation $\tau:U(\gg)\to\Mat_{m+1}(\KK)$ of the universal enveloping algebra $U(\gg)$. Since the group $\GG_a^n$ is commutative, the algebra $U(\gg)$ is isomorphic to the polynomial algebra
$\KK[x_1,\dots,x_n]$, where  $\gg$ is identified with the subspace $\langle x_1,\dots,x_n\rangle$.
The algebra $R:=\tau(U(\gg))$ is isomorphic to the factor algebra $U(\gg)/\Ker\tau$, where
$\Ker \tau =\{y\in U(\gg) \, : \, \tau(y)=0\}$.
As $\tau(x_1),\ldots,\tau(x_n)$ are commuting
nilpotent operators, the algebra $R$ is finite-dimensional and local. Let us denote by $X_1,\dots, X_n$ the images of the elements $x_1,\dots,x_n$
in~$R$. Then the maximal ideal of $R$ is $\mm:=(X_1,\dots, X_n)$.
Clearly, the subspace $U:=\tau(\gg)=\langle X_1,\ldots,X_n\rangle$ generates the algebra $R$.

Assume that $\KK^{m+1}$ is a cyclic $\GG_a^n$-module with a cyclic vector $v$, i.e., $\langle\rho(\GG_a^n)v\rangle=\KK^{m+1}$.
The subspace $\tau(U(\gg))v$ is $\gg$-  and $\GG_a^n$-invariant; it contains the vector $v$ and therefore coincides with the space $\KK^{m+1}$. Let $I=\{y\in U(\gg)\, : \, \tau(y)v=0\}$. Since the vector $v$ is cyclic, the ideal $I$ coincides with
$\Ker\tau$, and we get identifications
$$
R \cong U(\gg)/I \cong \tau(U(\gg))v=\KK^{m+1}.
$$
Under these identifications the action of an element $\tau(y)$ on $\KK^{m+1}$ corresponds to the operator of multiplication by $\tau(y)$ on the factor algebra $R$, and the vector $v\in\KK^{m+1}$ goes to the residue class of unit. Since
$\GG_a^n = \exp(\gg)$, the $\GG_a^n$-action on $\KK^{m+1}$ corresponds to the multiplication by elements of $\exp(U)$ on $R$.

\smallskip

Conversely, let $R$ be a local $(m+1)$-dimensional algebra with a maximal ideal $\mm$, and $U\subseteq \mm$ be a subspace that
generates the algebra $R$. Fix a basis $X_1,\ldots,X_n$ in $U$. Then $R$ admits a presentation $\KK[x_1,\ldots,x_n]/I$,
where $I$ is the kernel of the homomorphism $\KK[x_1,\ldots,x_n]\to R, \, x_i \mapsto X_i$.
These data define a faithful representation $\rho$ of the group $\GG_a^n:=\exp(U)$ on the space $R$: the operator $\rho((a_1,\dots,a_n))$ acts as multiplication by the element $\exp(a_1X_1+\dots+a_nX_n)$. Since $U$ generates $R$, one checks that the representation is cyclic with unit in $R$
as a cyclic vector.

Summarizing, we get the following result.

\begin{theorem} \cite[Theorem~2.14]{HT}. \label{T1}
The correspondence described above establishes a bijection between
\begin{enumerate}
\item[$(1)$] equivalence classes of faithful cyclic rational representations $\rho:\GG_a^n\to\GL_{m+1}(\KK)$;
\item[$(2)$] isomorphism classes of pairs $(R,U)$, where $R$ is a local $(m+1)$-dimensional algebra
with a maximal ideal $\mm$ and $U$ is an $n$-dimensional subspace of $\mm$ that generates the algebra $R$.
\end{enumerate}
\end{theorem}

\begin{remark}
Let $\rho:\GG_a^n\to\GL_{m+1}(\KK)$ be a faithful cyclic rational representation.
The set of cyclic vectors in $\KK^{m+1}$ is an open orbit of a commutative algebraic
group $C$ with  $\rho(\GG_a^n)\subseteq C \subseteq \GL_{m+1}(\KK)$, and the complement
of this set is a hyperplane. In our notation, the group $C$ is the extension
of the commutative unipotent group $\exp(\mm)$ by scalar matrices.
\end{remark}

A faithful linear representation $\rho:\GG_a^n\to\GL_{m+1}(\KK)$
determines an effective action of the group $\GG_a^n$ on the projectivization $\PP^m$ of the space $\KK^{m+1}$.
Conversely, let $G$ be a connected affine algebraic group with the trivial Picard group, and $X$ be a normal
$G$-variety. By \cite[Section~2.4]{KKLV}, every line bundle on $X$ admits a $G$-linearization.
Moreover, if $G$ has no non-trivial characters, then a $G$-linearization is unique.
This shows that every effective $\GG_a^n$-action on $\PP^m$ comes from a (unique) faithful rational
$(m+1)$-dimensional $\GG_a^n$-module.

An effective $\GG_a^n$-action on $\PP^m$ is generically transitive if and only if $n=m$.
In this case the corresponding $\GG_a^n$-module is cyclic.
It terms of Theorem~\ref{T1} the condition $n=m$ means $U=\mm$, and we obtain the following theorem.

\begin{theorem}\cite[Proposition~2.15]{HT} \label{T2}
There is a one-to-one correspondence between:
\begin{enumerate}
\item[$(1)$] equivalence classes of generically transitive $\GG_a^n$-actions on $\PP^n$;
\item[$(2)$] isomorphism classes of local $(n+1)$-dimensional algebras.
\end{enumerate}
\end{theorem}

In this correspondence we realize the space $\PP^n$ as the projectivization $\PP(R)$ of the algebra $R$, and the $\GG_a^n$-action is given as the action of a group of invertible elements of $R$ by multiplication. The complement of the open $\GG_a^n$-orbit on $\PP(R)$ is the hyperplane $\PP(\mm)$. More precisely, let us give an algebraic interpretation of the $\GG_a^n$-orbit structure on $\PP^n$.
Recall that two elements $a$ and $b$ of an algebra $R$ are {\it associated} if there exists an invertible element $c\in R$ such that $a=bc$.

\begin{proposition} \label{rem1}
The correspondence of Theorem~\ref{T2} determines a bijection between $\GG_a^n$-orbits on $\PP^n$ and association classes of nonzero elements in the algebra $R$.
\end{proposition}

\begin{proof}
The $\GG_a^n$-orbit of the point $[x]\in \PP(R)$, where $x\in R\setminus \{0\}$, coincides with
$[\KK^{\times}\cdot\exp(\mm)x]\subset \PP(R)$. Further, the set $R^{\times}$ of invertible elements of the algebra $R$
is $\KK^{\times}\oplus\mm$.  Since $\exp(\mm)=1+\mm$, we get $\KK^{\times}\cdot\exp(\mm)x=R^{\times}x$.
\end{proof}

\begin{remark}
The correspondence of Theorem~\ref{T2} together with classification results of \cite{ST}
yields that the number of equivalence classes of generically transitive $\GG_a^n$-actions
on $\PP^n$ is finite if and only if $n\le 5$ \cite[Section~3]{HT}, see also \cite{Ma}.
\end{remark}

%
\section{Modality of actions on projective spaces}
\label{sec2}

Suppose that an affine algebraic group $G$ acts regularly on an irreducible algebraic variety $X$. It is well known \cite[Section~1.4]{VP} that there is a non-empty open subset $W\subseteq X$ consisting of $G$-orbits of maximal dimension. Denote by $d(G,X)$ the codimension of the orbit $Gx$ in $X$, where $x\in W$. It follows from Rosenlicht's Theorem \cite[Section~2.3]{VP} that $d(G,X)$ equals the transcendency degree of the field $\KK(X)^G$ of rational invariants on the variety $X$. The condition $d(G,X)=0$ means that the $G$-action is generically transitive.

The {\it modality} $\modd(G,X)$ of a $G$-action on $X$ is the maximal value of $d(G,Y)$ over all irreducible $G$-invariant subvarieties $Y\subseteq X$. In particular, we have $\modd(G,X)=0$ if and only if $X$ contains a finite number of $G$-orbits.

Consider an action of the group $\GG_a^n$ on the space $\PP^m$ corresponding to the pair $(R,U)$ as in Section~\ref{sec1}.
Since the representation $d\rho$ of the Lie algebra $\gg = \text{Lie}(\GG_a^n)$ on the space $\KK^{m+1}$ is given
by multiplication by elements of $U$ on~$R$, we conclude that the isotropy subalgebra $\gg_x$ of an element $x\in R$ coincides with the subspace
$$
\Ann_U(x) \ = \ \{y\in U\, : \, xy=0\}.
$$
In particular, we have $\gg_x=0$ for every $x\in R\setminus\mm$. Denote by $C_{k,U}$ the Zariski closure of the subset $\{x\in\mm\, :\, \dim\Ann_U(x)=k\}\subseteq\mm$. Then the typical $\GG_a^n$-orbit on each irreducible component of $C_{k,U}$ is of dimension $n-k$. This gives a first interpretation of modality.

\begin{proposition} \label{pr1}
The following equality holds
$$
\modd(\GG_a^n, \PP^m)=\max_k (\dim C_{k,U}+k) - (n+1).
$$
\end{proposition}

In the sequel we consider generically transitive $\GG_a^n$-actions on $\PP^n$, i.e., assume $n=m$ and $U=\mm$. Let us introduce some notation: $\Ann(x):=\Ann_{\mm}(x)$, $C_k:=C_{k,\mm}$ and $\modd(R):=\modd(\GG_a^n,\PP^n)$. Clearly,  $0\le\modd(R)\le n-1$. For each $n$, Hassett and Tschinkel have proved that the extreme value $0$ is
achieved for a unique algebra $R$, see also Corollary~\ref{mod0} below. In Proposition~\ref{modn-1} we prove that
the extreme value $n-1$ is also achieved in a unique way.

\begin{example} \label{exam1}
Let $R=\KK[x_1,\dots,x_s]/I_{N+1}$, where the ideal $I_{N+1}\subset\KK[x_1,\dots,x_s]$ is generated by all  monomials of degree $N+1$. The annihilator of a polynomial $f(X_1,\dots,X_s)\in R$ coincides with the linear span of all monomials of degree $\ge N+1-k$, where $k$ is degree of the lowest term of $f$. Therefore all non-empty subsets $C_k$ are linear spans of monomials of degree greater or equal to a given value, and
$$
\modd(R)=\sum_{i=[\frac{N+1}{2}]}^N {s+i-1\choose s-1} -\sum_{i=1}^{[\frac{N}{2}]} {s+i-1\choose s-1}-1.
$$
In particular, $N=1$ and $s=n$ implies $\modd(R)=n-1$.
\end{example}

The operator of multiplication by $x\in\mm^i$ maps $\mm^j$ in $\mm^{i+j}$. Thus
$\dim\Ann(x)\ge\dim\mm^j-\dim\mm^{i+j}$, and we get

\begin{lemma} \label{pr2}
$\modd(R)\ge \max_{i,j}(\dim\mm^i+\dim\mm^j-\dim\mm^{i+j})-(n+1)$.
\end{lemma}

Set $r_i:=\dim\mm^i-\dim\mm^{i+1}$. In particular,  $r_0=1$. Suppose $\mm^N\ne 0$ and $\mm^{N+1}=0$. Then $n=r_1+r_2+\dots+r_N$.
Using Lemma~\ref{pr2} with $j=1$, we obtain

\begin{proposition} \label{pr3}
$\modd(R)\ge \max_i r_i-1$.
\end{proposition}

 Example~\ref{exam1} shows that the modality can significantly exceed the value $\max_i r_i-1$.

\begin{corollary}
\label{modgen}
If $\modd(R)=l$, then the algebra $R$ can be generated by $l+1$ elements and $\dim\mm^N\le l+1$.
\end{corollary}

\begin{corollary}\cite[Proposition~3.7]{HT}
\label{mod0}
For any $n>0$ there exists a unique $(n+1)$-dimensional local algebra of modality zero; it is
isomorphic to $\KK[x]/(x^{n+1})$.
\end{corollary}

\begin{proof}
By Corollary~\ref{modgen}, an algebra of modality zero is generated by one element.
\end{proof}

Let us recall that the sequence $(r_0,r_1,\ldots,r_N)$ is called the {\em Hilbert-Samuel sequence}
of the algebra $R$. This sequence does not determine modality of the algebra. Indeed, consider the algebras
$$
\KK[x,y]/(x^4,xy,y^3) \ \ \text{and} \ \ \KK[x,y]/(x^4, y^2, x^2y).
$$
They share the same Hilbert-Samuel sequence $(1,2,2,1)$. In the next section it is shown that the first algebra has modality one, while the second one has modality two.

Let $\Ann(\mm)=\{x\in\mm\, :\, yx=0 \ \text{for}\ \text{all}\ y\in\mm\}$. Then
the subspace $\PP(\Ann(\mm))$ coincides with the set of  $\GG_a^n$-fixed points in $\PP^n$,
and the inequality $\modd(R)\ge r_N-1$ can be strengthened as
$\modd(R)\ge\dim\Ann(\mm)-1$. It is known that a local finite-dimensional algebra $R$ is Gorenstein if and only if $\dim\Ann(\mm)=1$ \cite[Proposition~21.5]{Ei}. Such algebras are assigned to generically transitive $\GG_a^n$-actions on $\PP^n$ with a unique fixed point.

\begin{proposition}
\label{modn-1}
For each $n>0$ there exists a unique $(n+1)$-dimensional local algebra of modality $n-1$; it is isomorphic to $\KK[x_1,\ldots,x_n]/(x_ix_j, \, 1 \le i \le j \le n)$.
\end{proposition}

\begin{proof}
Suppose there is a $\GG_a^n$-action on $\PP^n$ with an $(n-1)$-parameter continuous orbit family. Then any orbit from this family is just a $\GG_a^n$-fixed point. The complement to the open orbit in $\PP^n$ is the hyperplane $\PP(\mm)$, so all the points in this hyperplane should be $\GG_a^n$-fixed. This condition is equivalent to the equality $\mm^2=0.$
\end{proof}

\begin{remark} \label{rems}
Since the association classes in the algebra $R$ coincide with orbits of the lineal algebraic group
$\KK^{\times}\cdot\exp(\mm)$, the classes of a given dimension form a finite union of irreducible
locally closed subvarieties in $R$. The codimension of a class in such a subvariety may be addressed
as the number of parameters in a continuous family of classes. The codimension is preserved under the
passage to projectivization $\PP(R)$, and by Proposition~\ref{rem1}, the maximal
number of parameters is the modality of $R$. Moreover, the closure of the association class of
an element $x\in R$ coincides with the principal ideal $Rx$. Thus we may speak about the number of
parameters in a continuous family of principal ideals and interpret $\modd(R)$ as the maximum
of these numbers.
\end{remark}

\begin{proposition}
\label{surhom}
If $I$ is a proper ideal of the algebra $R$, then
$$
\modd(R/I) \le \modd(R).
$$
\end{proposition}

\begin{proof}
The canonical homomorphism $R \to R/I$ induces a homomorphism of groups
$$
\GG_a^n:=\exp(\mm) \to \GG_a^r:=\exp(\mm/I), \quad \text{where}  \quad r=\dim(R/I)-1,
$$
and a surjective morphism of varieties $\PP(R) \setminus \PP(I) \to \PP(R/I)$.
The latter morphism is equivariant with respect to the $\GG_a^n$- and $\GG_a^r$-actions on $\PP(R)$ and $\PP(R/I)$ respectively.
Assume that we have a continuous $k$-parameter family of $\GG_a^r$-orbits on $\PP(R/I)$. Some irreducible component of the inverse image of this family projects to the family dominantly. Then typical $\GG_a^n$-orbits on this component form at least $k$-parameter family on $\PP(R) \setminus \PP(I)$. This completes the proof.
\end{proof}

We finish this section by introducing a combinatorial way to estimate modality of a local $(n+1)$-dimensional algebra $\KK[x_1,\dots,x_s]/I$, where $I$ is a monomial ideal. Denote by $\mathcal N$ the set of non-constant monomials in $x_1,\dots, x_s$ that do not belong to $I$. To each subset $\mathcal{L}\subseteq\mathcal N$  assign a subset
$S(\mathcal L)=\{a\in\mathcal N\, : \, ab\in I \ \text{for}\ \text{any} \ b\in\mathcal L\}$.
The subset $S(\mathcal L)$ consists of monomials annihilating all the elements of the factor algebra with support at $\mathcal L$. It implies the inequality
$$
\modd(R)\ge \max_{\mathcal L}(|\mathcal L|+|S(\mathcal L)|) -(n+1),
$$
which can be strict. For example, with the algebra $\KK[x_1,x_2]/(x_1^2,x_2^2)$ we get $1>0$.

%
%
\section{Algebras of modality one}
\label{sec22}

In this section we classify local algebras corresponding to generically transitive $\GG_a^n$-actions on $\PP^n$ of modality one.

\begin{theorem} \label{Tclas}
Local finite-dimensional algebras of modality one form the following list:
$$
A_{a,b}=\KK[x,y]/(x^{a+1},y^{b+1},xy), \ a\ge b\ge 1; \quad B_{a,b}=\KK[x,y]/(xy, x^a-y^b), \ a\ge b\ge 2 ;
$$

$$
C_a=\KK[x,y]/(x^{a+1}, y^2 - x^3),  \ a \ge 3; \quad C_a^1=\KK[x,y]/(x^{a+1}, y^2 - x^3, x^ay), \ a\ge 3;
$$

$$
C_a^2=\KK[x,y]/(x^{a+1}, y^2 - x^3, x^{a-1}y),  \ a \ge 3; \quad C_a^3=\KK[x,y]/(y^2 - x^3, x^{a-2}y), \ a\ge 4;
$$

$$
D=\KK[x,y]/(x^3, y^2); \quad E=\KK[x,y]/(x^3, y^2 ,x^2y).
$$
These algebras are pairwise non-isomorphic.
\end{theorem}

\begin{proof}
It follows from Corollary~\ref{mod0} that all algebras of Theorem~\ref{Tclas} have positive modality.
Let us show that their modality does not exceed one.

Every  element of the algebra $R=A_{a,b}$ has the form
$$
\alpha_k x^k + \alpha_{k+1} x^{k+1} + \ldots + \beta_s y^s +\beta_{s+1} y^{s+1} +\ldots \ = \
$$
$$
 \ = \
(\alpha_k x^k + \beta_s y^s)(1 + \widetilde{\alpha_{k+1}}x + \ldots + \widetilde{\beta_{s+1}}y + \ldots),
$$
where the second factor is invertible. For every $1\le k \le a$ and $1\le s \le b$ define a morphism
$$
\phi_{k,s} \colon \PP^1 \to \PP(R), \quad  \phi_{k,s}([\alpha:\beta])=[\alpha x^k+ \beta y^s].
$$
Since every non-open $\GG_a^n$-orbit in $\PP(R)$ intersects $\phi_{k,s}(\PP^1)$ for some $k$ and $s$,
modality of the algebra $A_{a,b}$ does not exceed one.
The algebra $B_{a,b}$ is a homomorphic image of $A_{a,b}$, and Proposition~\ref{surhom} implies that modality of $B_{a,b}$ does not exceed one.

\smallskip

For other types of algebras the following lemma is useful.

\begin{lemma}
\label{basmod}
Suppose that a linear basis $v_1,\ldots,v_n$ of the ideal $\mm$ of a local algebra $R$ satisfies the following conditions:
\begin{enumerate}
\item
the element $v_s$ does not belong to the principal ideal $(v_{s+1})$ for every $s\ge 1$;
\item
for all $v_s$ and $v_{s+p}$, $p \ge 2$, there exists a vector $w_{s,p} \in \mm$ such that
$v_{s+p} = v_s w_{s,p}$ and $v_{s+1} w_{s,p} = 0$ or $v_r$, $r \ge s + p +1$.
\end{enumerate}
Then the modality of the algebra $R$ does not exceed one.
\end{lemma}

\begin{proof}
Consider an element
$$
b \ = \ \alpha_0 v_s + \sum_{p\ge 1} \alpha_p v_{s+p}, \quad \text{with} \quad \alpha_0 \ne 0.
$$
Let us show that there exists an element
$$
c = 1 + \beta_2 w_{s,2} + \ldots + \beta_{n-s} w_{s, n-s}
$$
such that $b = (\alpha_0 v_s + \alpha_1 v_{s+1}) c$.
Indeed, write down the equality
$$
\alpha_0 v_s + \alpha_1 v_{s+1} + \ldots + \alpha_{n-s} v_n \ = \
(\alpha_0 v_s + \alpha_1 v_{s+1})(1 + \beta_2 w_{s,2} + \ldots + \beta_{n-s}w_{s,n-s}),
$$
and consider the coefficients at $v_{s+2}, \ldots, v_n$ on the right. They are equal to $\alpha_2,\ldots,\alpha_{n-s}$ respectively, so one may find the values of $\beta_2,\ldots,\beta_{n-s}$ successively.
The element $c$ is invertible, so again every non-open $\GG_a^n$-orbit in $\PP(R)$ intersects one of the curves
$\phi_s(\PP^1)$ with $1\le s \le n-1$, where
$$
\phi_s \colon \PP^1 \to \PP(R), \quad  \phi_s([\alpha:\beta])=[\alpha v_s+ \beta v_{s+1}].
$$
The assertion follows.
\end{proof}

Let us indicate a basis satisfying the conditions of Lemma~\ref{basmod} for the algebra~$C_a$:
$$
 x,\, y,\, x^2,\, xy,\, x^3,\, x^2y,\, \ldots, x^a,\, x^{a-1}y,\, x^ay,
$$
and for the algebra $D$:
$$
x,\ y,\ x^2,\ xy,\ x^2y.
$$
The algebras of types $C_a^1$, $C^2_a$, $C_a^3$ are factor algebras of $C_a$ and $E$ is a factor algebra of $D$, so their modality cannot exceed one.

\smallskip

Now we prove that any algebra of modality one is included in the list of Theorem~\ref{Tclas}.

\begin{lemma}
Modality of the algebra $H := \KK[x,y]/(x^4, y^2, x^2y)$ is greater than one.
\end{lemma}

\begin{proof}
One checks that typical $\GG_a^5$-orbits on the projective 3-space $\PP(\langle x^2, x^3, y, xy\rangle)$ are one-dimensional.
\end{proof}

\begin{lemma}
\label{reduce}
Let $R=\KK[x,y]/I$, where the ideal $I$ is generated by polynomials not containing any of the terms $x,x^2,x^3,y,xy$. Then modality of $R$ is greater than one.
\end{lemma}

\begin{proof}
By assumptions, the algebra $H$ is a factor algebra of $R$.
\end{proof}

Let $R$ be a local finite-dimensional algebra of modality one. Using Corollary ~\ref{modgen}, one can assume that it has the form $\KK[x,y]/I$, where every term of a polynomial from $I$ has degree $\ge 2$. Consider the algebra $R'=R/\mm^3$. It is easy to see that there exists a linear change of variables taking $R'$ to one of the following non-isomorphic algebras:
$$
\KK[x,y]/(x^2,xy,y^2),\ \KK[x,y]/(x^3,xy,y^2),\ \KK[x,y]/(x^2,y^2),
$$
$$
\KK[x,y]/(x^3,xy,y^3),\ \KK[x,y]/(x^3,x^2y,y^2).
$$

\smallskip

{\it Case 1}: $R'=\KK[x,y]/(x^2,xy,y^2)$. In this case we have $\mm^2=\mm^3$, so $R=R'=A_{1,1}$.

\smallskip

{\it Case 2}: $R'=\KK[x,y]/(x^3,xy,y^2)$. In this case $xy, y^2\in\mm^3$ and the factor space $\mm^i/\mm^{i+1}$ is spanned by the class of the element $x^i$ with $i\ge 3$. The elements of the form $x^i$ are linearly independent in $R$, hence the elements $1,y,x,x^2,\dots,x^a$, $a\ge 2$, form a basis of this algebra. Assume that $xy=\sum_{i\ge 3} \alpha_ix^i$. Replace the basis element $y$ with $y'=y-\sum_{i\ge 3} \alpha_ix^{i-1}$ to obtain $xy'=0$ and $(y')^2=\sum_{j\ge 3} \beta_jx^j$. Suppose that the last expression is not zero  and $\beta_s$ is the first nonzero coefficient. If $s<a$, then multiplying the equality by $x^{a-s}$, we get $x^a=0$, a contradiction. It means that $(y')^2=0$ or $(y')^2=x^a$, and the algebras $A_{a,1}$ ($a\ge 2$) and $B_{a,2}$ ($a\ge 3$) appear. Algebras of the second type are Gorenstein unlike the others. Algebras of the same type with different indices have different dimensions.

\smallskip

{\it Case 3}: $R'=\KK[x,y]/(x^2,y^2)$. In this case $x^2,y^2\in\mm^3$ and every monomial of degree three is in $\mm^4$, hence $\mm^3=\mm^4=0$,  $R=R'$. It is easy to check that the algebra $R'$ is isomorphic to $B_{2,2}$.

\smallskip

{\it Case 4}: $R'=\KK[x,y]/(x^3,xy,y^3)$. In this case we have $xy\in\mm^3$ and $R$ is a linear span of the elements $1,x,\dots,x^a,y,\dots,y^b$, $a\ge b\ge 2$. Assume that $xy=\sum_{i\ge 3}\alpha_ix^i+\sum_{j\ge 3}\beta_jy^j$.
Let us show that there exists a change of variables providing the condition $xy=0$. Indeed, take a basis
$x'=x+\sum_{k\ge 2}\phi_ky^k$ and $y'=y+\sum_{s\ge 2}\psi_sx^s$ and write down
$$
x'y'=\sum_{i\ge 3}\alpha_ix^i+\sum_{j\ge 3}\beta_jy^j+\sum_{k\ge 2}\phi_ky^{k+1}+
\sum_{s\ge 2}\psi_sx^{s+1}+\sum_{k,s\ge 2}\phi_k\psi_s x^sy^k.
$$
Here the monomial $x^sy^k$ is decomposed as the sum of terms $x^i$ and $y^j$ of degree at least $k+s+2$ each.
Now one can consider the coefficients at $x^3$, $y^3$, $x^4$, $y^4,\dots$ in the expression for $x'y'$ and obtain $\psi_2$, $\phi_2$, $\psi_3$, $\phi_3,\dots$ one by one. Therefore, we can further assume that $xy=0$.

Suppose there is an element $\xi_rx^r+\dots+\xi_ax^a+\mu_py^p+\dots+\mu_by^b$, $\xi_r\ne 0$, in the ideal $I$. If $r<a$, then we can multiply both parts by $x^{a-r}$ and come to a contradiction with the condition $x^a\ne 0.$ In the same way it can be shown that $p=b$. Thus the algebras
$A_{a,b}$, $a\ge b\ge 2$, and $B_{a,b}$,
$a\ge b\ge 3$, appear. The inequalities are caused by linear independency of the classes of the elements $x^2$ and $y^2$ in $\mm^2/\mm^3$. The algebras from the second family are Gorenstein unlike the others. The algebras of the same type with different indices have different Hilbert-Samuel functions.

\smallskip

{\it Case 5}: $R'=\KK[x,y]/(x^3,x^2y,y^2)$. In this case we have $y^2=\sum_{i\ge 3}\alpha_ix^i+\sum_{j\ge 2}\beta_jx^jy$. A change
$y'=y+\sum_{p\ge 2} \phi_px^p$ leads to the equation
$$
(y')^2=\sum_{i\ge 3}\alpha_ix^i+\sum_{j\ge 2}\beta_jx^jy+2\sum_{p\ge 2} \phi_px^py+\sum_{p,s\ge 2}\phi_p\phi_sx^{p+s}.
$$
Setting the coefficients at $x^2y,x^3y,\dots$ equal zero permits to obtain $\phi_2,\phi_3,\dots$ one by one. Therefore one can assume that $y^2=\sum_{i\ge 3}\alpha_ix^i$. It is easy to see that there exists a root of any integer order from an element  $\alpha_0 + \alpha_1 x + \alpha_2 x^2 + \ldots$, where $\alpha_0\ne 0$, in $R$. If $y^2 = x^e(\alpha_e + \alpha_{e+1}x + \ldots)$ and $\alpha_e \ne 0$, then replace $y$ by $v^{-1}y$, where $v^2 = \alpha_e + \alpha_{e+1}x + \ldots$, to obtain $y^2=x^e$. Now there are the variants $y^2=0$ or $y^2=x^e$. If the elements $x^a$ and $x^cy$ are nonzero and $x^{a+1}=x^{c+1}y=0$, then $a\ge c\ge 1$ è $a\ge 2$ (since the elements $x^2$ and $xy$ are not in $\mm^3$). Also the second possibility requires the restrictions $a\ge e\ge 3$ and $c+e\ge a$, since $y^2x^{c+1}=x^{c+e+1}=0$.

\smallskip

{\it Case 5.1}: $y^2 = 0$. There are also the conditions $x^{a+1}=0$ and $x^{c+1}y=0$. If there is no other relation in $R$, then by Lemma~\ref{reduce} we obtain the algebras $D$ and $E$.
Suppose that we have another equality
$\sum_{i\ge i_0} \xi_ix^i+\sum_{j\ge j_0} \mu_jx^jy=0$. Then $i_0\ge 3$ and $j_0\ge 2$, since the elements $x$ and $y$ are linearly independent mod $\mm^2$, and the elements $x^2$ and  $xy$ are linearly independent mod $\mm^3$. If all $\xi_i=0$ (resp. $\mu_j=0$), then we obtain a contradiction multiplying by $x^{c-j_0}$ (resp. by $x^{a-i_0}$). If $i_0\le c$, then we get a contradiction multiplying by $y$. Hence $i_0\ge c+1$. Moreover, if $a-i_0\ne c-j_0$, then multiplying by an appropriate power of $x$ again leads to a contradiction. It means that $i_0=a-d$ and $j_0=c-d$ for some $d\ge 0$; here $a-d>c$. Assume that $d$ is maximal over all relations. So there is a relation  $x^{a-d}\chi_1 - x^{c-d} \chi_2 y = 0$, where $\chi_i = \alpha_{i0} + \alpha_{i1} x + \ldots$,
$\alpha_{i0} \ne 0$, or equivalently, $x^{a-d} - x^{c-d}\chi y = 0$, where $\chi := \chi_1^{-1} \chi_2$. The replacement of $y$ by $\chi y$ provides the relation $x^{a-d} - x^{c-d}y = 0$. One can subtract the relation of this form from any other relation with corresponding lower terms to cancel the lower terms or to obtain a contradiction. From Lemma~\ref{reduce} and the inequality $a-d > c > d+1 \ge 1$ it follows that $a-d = 3 $ and $c=2$. In this case $d=0$, $a=3$, $c=2$, and our algebra can be written as $\KK[x,y]/(y^2, x^3y, x^3-x^2y)$. Replace the variable  $x$  by $-3x+y$ to show that this algebra is isomorphic to $D$.

\smallskip

{\it Case 5.2}: here $R$ is a factor algebra of the algebra $\KK[x,y]/(x^{a+1},y^2-x^e,x^{c+1}y)$. If $R$ coincides with this algebra, then Lemma~\ref{reduce} reduces the proof to the case $e=3$. The inequalities $a \ge c \ge a-e$ imply four possible variants: $c = a,\, a-1,\, a-2,\, a-3\,$. In the last case it follows from $c\ge 1$ that $a\ge 4$ and we get the algebras $C_a$, $C_a^1$, $C_a^2$ and $C_a^3$ respectively.

Assume there is another relation $\sum_{i\ge i_0} \xi_ix^i+\sum_{j\ge j_0} \mu_jx^jy=0$. As in the previous case one shows that $i_0\ge 3$, $j_0\ge 2$ and $a-i_0=c-j_0=d$. Again set the value of $d$ maximal over all relations; thus the equation of described form implies all the relations in the algebra $R$. Multiply the equation by $y$ to obtain another equation with lower terms $x^{c-d+e}$ and $x^{a-d}y$. It is compatible with the others whenever there are extra restrictions: two inequalities $c-d+e>a$ and $a-d>c$ or one equality $a-(c-d+e)=c-(a-d)$. The last formula is equivalent to
$e=2(a-c)$. Here $e\ge 4$, hence $a-c\ge 2$. It means that $a-d \ge c-d+2$; therefore we have $a-d\ge 4$, $c-d \ge 2$. From Lemma~\ref{reduce} it follows that there are no algebras of modality one in this case.

Thus the restrictions $c-d+e>a$ and $a-d>c$ hold. Lemma~\ref{reduce} reduces the remainder to the following cases:

\smallskip

{\it Case 5.2.1}: $e=3$. From the inequalities $c+3-d>a$ and $a-d>c$ it follows that $3-2d>0$, hence $d=0,1$. If $d=1$, then $c+2>a>c+1$, and we come to a contradiction. If $d=0$, then $c+3>a>c$; in other words, we have $c=a-1$ or $c=a-2$. If the second possibility holds, then
$a=c+2\ge 2+d+2 = 4$. Denote the obtained algebras by $D_a^1$ and $D_a^2$ respectively. We are going to prove that they are isomorphic to $C_{a-1}$ and $C_a^3$.

\smallskip

{\it Case 5.2.1.1}: $D_a^1\cong C_{a-1}$ $(a\ge 4)$ and $D_3^1 \cong D$. Consider the following variable  change in $C_{a-1}$: $x'=ux+vy$, $y'=wx^2+zy$, where $u,v,w,z$ are some invertible polynomials in $x$. The first required equality $(x')^ay'=0$ holds automatically. The second condition $(x')^a=(x')^{a-1}y\ne 0$ gives an equality for the constant terms: $av_0=z_0$. The third relation $(x')^3=(y')^2$ provides the equation
$$
x^3(u^3-z^2)+x^4(3uv^2-w^2)+x^2y(3u^2v-2wz)+x^3yv^3=0.
$$
Since the $u,v,w,z$ do not contain $y$, one obtains all the coefficients of $u,v,w,z$ considering the coefficients at $x^3$, $x^2y$, $x^4,$ $x^3y,\dots$. Thus we determine the required generators and prove the isomorphism.

\smallskip

{\it Case 5.2.1.2}: $D_a^2\cong C_a^3.$ The proof is similar. Consider $x'=ux+vy$ and $y'=wx^2+zy$ in $C_a^3$, where $u,v,w,z$ are some invertible polynomials in $x$. The condition $(x')^{a-1}y'=0$ holds automatically; the relation $(x')^a=(x')^{a-2}y\ne 0$ results in the equality for the constant terms $u_0^3=u_0w_0+(a-2)v_0z_0$; the last equality $(x')^3=(y')^2$ provides the relation
$$
x^3(u^3-z^2)+x^4(3uv^2-w^2)+x^2y(3u^2v-2wz)+x^3yv^3=0,
$$
which determines $u,v,w,z$ like in the previous case.

\smallskip

{\it Case 5.2.2}: $a-d=3$. Here we obtain $a=3$, hence $e=3$, and this case reduces to 5.2.1.

\smallskip

Now we have to prove that the algebras $C_a$, $C_a^1$, $C_a^2$, $C_a^3$, $D$ and $E$ are not pairwise isomorphic. The algebras $C_a$, $C_a^3$ and $D$ are Gorenstein. To prove that they are not isomorphic let us calculate their dimensions:
$$
\dim C_a \ = \ 2(a+1) \ \ge \ 8, \quad \dim C_a^3 \ = \ 2a-1 \ \ge \ 7, \quad \dim D \ = \ 6.
$$

For non-Gorenstein algebras dimensions are
$$
\dim C_a^1 \ = \ 2a+1 \ \ge \ 7, \quad
\dim C_a^2 \ = \ 2a \ \ge \ 6, \quad \dim E \ = \ 5.
$$
This completes the proof of Theorem~\ref{Tclas}.
\end{proof}

\begin{corollary}
For $n\ge 5$ there are $(n+1)$ pairwise non-isomorphic local $(n+1)$-dimensional algebras
of modality one, while for $n=2$ we have $1$ algebra, for $n=3$ there are $2$ algebras, and for
$n=4$ there are $4$ algebras.
In particular, for any $n>0$ the number of equivalence classes of generically transitive $\GG_a^n$-actions of modality one on $\PP^n$ is finite.
\end{corollary}

Let us show that there is an infinite family of pairwise non-isomorphic 7-dimensional local algebras of modality two. Consider algebras with Hilbert-Samuel sequence $(1,3,3)$. Here multiplication is determined by a bilinear symmetric map $\mm/\mm^2\times\mm/\mm^2\to\mm^2$. Such maps form a 18-dimensional space; the group  $\GL(3)\times\GL(3)$ acts here with a one-dimensional inefficiency kernel. It means that there are infinitely many generic pairwise non-isomorphic algebras of this type. Let us prove that modality of a generic algebra equals two. By Proposition ~\ref{pr1}, $\modd(R)=\max_k(\dim C_k+k)-(n+1)$; here the annihilator of a generic point $x\in\mm$ coincides with $\mm^2$, so $k=3$ and $C_3=\mm$. When $k=4$ (resp. 5), the dimension of $C_k$ decreases by 1 (resp. at least by 2). Finally, for $k=6$ the space $C_k$ coincides with $\mm^2$. One concludes that $\modd(R)=(6+3)-7=2$.

\smallskip

Note that for every monomial ideal $I$ the finite-dimensional algebra $\KK[x,y]/I$ determines the Young diagram with boxes assigned to all monomials not in $I$. Conversely, to every Young diagram one assigns a local finite-dimensional monomial algebra with two generators. The diagrams $(n)$ and $(1,\ldots,1)$ correspond to modality zero.

\bigskip

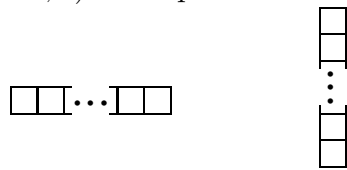
\begin{picture}(100,10)
\put(100,20){\line(1,0){20}}
\put(100,20){\line(0,1){10}}
\put(100,30){\line(1,0){20}}
\put(110,20){\line(0,1){10}}
\put(120,20){\line(0,1){10}}
\put(120,20){\line(1,0){3}}
\put(120,30){\line(1,0){3}}
\put(125,24){\circle*{2}}
\put(130,24){\circle*{2}}
\put(135,24){\circle*{2}}
\put(140,20){\line(1,0){20}}
\put(150,20){\line(0,1){10}}
\put(140,20){\line(0,1){10}}
\put(140,20){\line(-1,0){3}}
\put(140,30){\line(-1,0){3}}
\put(140,30){\line(1,0){20}}
\put(160,20){\line(0,1){10}}
\end{picture}
\quad
\begin{picture}(60,50)
\put(100,0){\line(0,1){20}}
\put(100,0){\line(1,0){10}}
\put(110,0){\line(0,1){20}}
\put(100,10){\line(1,0){10}}
\put(100,20){\line(1,0){10}}
\put(100,20){\line(0,1){3}}
\put(110,20){\line(0,1){3}}
\put(100,40){\line(0,-1){3}}
\put(110,40){\line(0,-1){3}}
\put(104,25){\circle*{2}}
\put(104,30){\circle*{2}}
\put(104,35){\circle*{2}}
\put(100,40){\line(0,1){20}}
\put(100,50){\line(1,0){10}}
\put(100,40){\line(1,0){10}}
\put(110,40){\line(0,1){20}}
\put(100,60){\line(1,0){10}}
\end{picture}

\bigskip

By Theorem~\ref{Tclas}, diagrams of modality one are exactly $(s,1,\ldots,1)$, $(2,2)$, $(3,2)$, $(2,2,1)$,
$(3,3)$ and $(2,2,2)$.

\bigskip

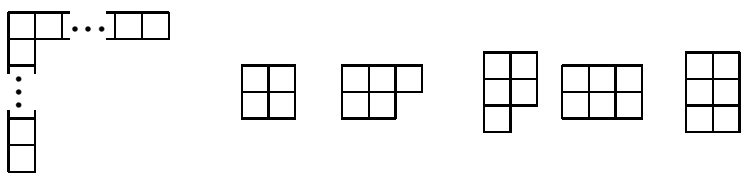
\begin{picture}(40,60)
\put(40,50){\line(1,0){20}}
\put(40,40){\line(0,1){20}}
\put(40,60){\line(1,0){20}}
\put(50,40){\line(0,1){20}}
\put(80,50){\line(0,1){10}}
\put(60,50){\line(0,1){10}}
\put(60,50){\line(1,0){3}}
\put(60,60){\line(1,0){3}}
\put(80,50){\line(-1,0){3}}
\put(80,60){\line(-1,0){3}}
\put(40,20){\line(1,0){10}}
\put(40,40){\line(1,0){10}}
\put(40,20){\line(0,1){3}}
\put(50,20){\line(0,1){3}}
\put(40,40){\line(0,-1){3}}
\put(50,40){\line(0,-1){3}}
\put(65,54){\circle*{2}}
\put(70,54){\circle*{2}}
\put(75,54){\circle*{2}}
\put(80,50){\line(1,0){20}}
\put(90,50){\line(0,1){10}}
\put(80,60){\line(1,0){20}}
\put(100,50){\line(0,1){10}}
\put(44,25){\circle*{2}}
\put(44,30){\circle*{2}}
\put(44,35){\circle*{2}}
\put(40,0){\line(0,1){20}}
\put(50,0){\line(0,1){20}}
\put(40,0){\line(1,0){10}}
\put(40,10){\line(1,0){10}}
\end{picture}
\qquad
\begin{picture}(20,20)
\put(60,20){\line(0,1){20}}
\put(60,20){\line(1,0){20}}
\put(70,20){\line(0,1){20}}
\put(60,30){\line(1,0){20}}
\put(60,40){\line(1,0){20}}
\put(80,20){\line(0,1){20}}
\end{picture}
\qquad
\begin{picture}(30,20)
\put(50,20){\line(0,1){20}}
\put(50,20){\line(1,0){20}}
\put(60,20){\line(0,1){20}}
\put(50,30){\line(1,0){30}}
\put(50,40){\line(1,0){30}}
\put(70,20){\line(0,1){20}}
\put(80,30){\line(0,1){10}}

\end{picture}
\qquad
\begin{picture}(30,20)
\put(75,20){\line(0,1){20}}
\put(75,20){\line(1,0){30}}
\put(85,20){\line(0,1){20}}
\put(75,30){\line(1,0){30}}
\put(75,40){\line(1,0){30}}
\put(95,20){\line(0,1){20}}
\put(105,20){\line(0,1){20}}
\end{picture}
\quad
\begin{picture}(20,30)
\put(0,15){\line(1,0){10}}
\put(0,15){\line(0,1){30}}
\put(0,25){\line(1,0){20}}
\put(10,15){\line(0,1){30}}
\put(20,25){\line(0,1){20}}
\put(0,35){\line(1,0){20}}
\put(0,45){\line(1,0){20}}
\end{picture}
\quad
\begin{picture}(20,30)
\put(40,15){\line(1,0){20}}
\put(40,15){\line(0,1){30}}
\put(40,25){\line(1,0){20}}
\put(50,15){\line(0,1){30}}
\put(60,15){\line(0,1){30}}
\put(40,35){\line(1,0){20}}
\put(40,45){\line(1,0){20}}
\end{picture}

\bigskip

\noindent It would be interesting to calculate the modality for an arbitrary Young diagram.
Another intriguing problem is to connect Theorem~\ref{Tclas} with the classification
of singularities of modality one in the sense of \cite{AGV}.

Let us also mention that finite-dimensional 2-generated local algebras occur in the study of local punctual Hilbert schemes and of pairs of commuting nilpotent matrices. Lately there were many works in this direction, see~\cite{BI} and references therein. Let $(A,B)$ be a pair
of commuting nilpotent $n\times n$ matrices with a fixed cyclic vector. If $R=\KK[A,B]$, then $\modd(R)$ may be interpreted as the maximal number of parameters in a continuous family of subspaces invariant for both $A$ and $B$. The variety of pairs of commuting nilpotent matrices with a fixed cyclic vector admits a natural action of the general linear group. It is shown in \cite[Theorem~1.9]{Na} that if we consider the canonical map from the punctual Hilbert scheme $\text{Hilb}^n(\mathbb{A}^2)$ to the symmetric power $\text{Sym}^n(\mathbb{A}^2)$, the quotient by the action mentioned above is the fibre $H[n]$ over the point $np$, where $p$ is a point on $\mathbb{A}^2$. Roughly speaking, the local punctual Hilbert scheme $H[n]$
parametrizes 2-generated local algebras of dimension $n$. The Hilbert-Samuel sequence defines a stratification of $H[n]$ by locally closed
subsets, see~\cite[Proposition~1.6]{Ia}. Applications of this stratification in linear algebra may be found in \cite{BI}.
The modality being another numerical invariant of a local algebra gives rise to other
stratification of $H[n]$, and it may be interesting to study its properties.


\section{$\GG_a^n$-actions on degenerate quadrics}
\label{sec3}

We return to the settings of Section~\ref{sec1}.
Given the projectivization $\PP^m$ of a rational $\GG_a^n$-module and a point $x\in\PP^m$,
the closure of the orbit $\GG_a^n\cdot x$ is a projective variety. Since the closure is $\GG_a^n$-invariant,
it inherits a (generically transitive) $\GG_a^n$-action.
Conversely, any normal projective $\GG_a^n$-variety can be equivariantly embedded into the projectivization of a rational $\GG_a^n$-module \cite[Section~1.2]{VP}.
If the action is generically transitive, then one may assume that this is the orbit closure of a cyclic vector.
Therefore any generically transitive $\GG_a^n$-action on a normal projective variety is determined by a pair $(R,U)$, see Theorem~\ref{T1}. This time the correspondence is not bijective, since the orbit closure is not necessarily normal and the equivariant embedding of a projective variety into the projectivization of a $\GG_a^n$-module is not unique. For example, the pair $(\KK[x]/(x^{m+1}), \langle x\rangle)$ determines a standard $\GG_a$-action on $\PP^1$ for all $m\ge 1$.

Let $R$ be a local $(m+1)$-dimensional algebra
with a maximal ideal $\mm$ and $U$ be an $n$-dimensional subspace of $\mm$ that generates the algebra $R$.
A pair $(R,U)$ is called an {\it $H$-pair} if $U$ is a hyperplane in $\mm$, or, equivalently, $m=n+1$. Such pairs correspond to generically transitive $\GG_a^n$-actions on hypersurfaces in $\PP^{n+1}$. The  degree of the hypersurface is called the {\it degree} of the $H$-pair.

An $H$-pair of degree 2 is called {\it quadratic}. In this case we get a generically transitive $\GG_a^n$-action of a projective quadric
$$
Q(n,k):=\{[z_0:z_1:\ldots:z_{n+1}] \ ; \ q(z_0,z_1,\ldots,z_{n+1})=0 \} \subset \PP^{n+1},
$$
where $q$ is a quadratic form of rank $k+2$ with $1\le k \le n$. In this notation the non-degenerate quadric $Q_n\subset\PP^{n+1}$ is $Q(n,n)$;
see \cite[Chapter~I, Exercise~5.12]{Ha} for basic geometric properties of projective quadrics.

\begin{example}
Consider the $H$-pair $(R,U)=(\KK[x]/(x^5), \langle x,x^2,x^4\rangle)$. Denote by $X_1$, $X_2$ and $X_3$
the images of the elements $x$, $x^2$ and $x^4$ in the algebra $R$. We have
$$
\exp(a_1X_1+a_2X_2+a_3X_3) \ = \
$$
$$
\ = \
1+a_1X_1+a_2X_2+a_3X_3+\frac{a_1^2}{2}X_2+\frac{a_2^2}{2}X_3+a_1a_2X_1^3+\frac{a_1^3}{6}X_1^3+\frac{a_1^2a_2}{2}X_3+
\frac{a_1^4}{24}X_3,
$$
and the $\GG_a^3$-orbit of the line $\langle 1\rangle$ has the form
$$
\left\{\left[\,1\,:\,a_1\,:\,a_2+\frac{a_1^2}{2}\,:\,a_1a_2+\frac{a_1^3}{6}\,:\,a_3+\frac{a_2^2}{2}+\frac{a_1^2a_2}{2}+\frac{a_1^4}{24}\,\right]
\, ; \, a_1,a_2,a_3\in\KK\right\}.
$$
In homogeneous coordinates $[z_0:z_1:z_2:z_3:z_4]$ on $\PP^4$ our pair determines a hypersurface given by the equation:
$$
3z_0^2z_3-3z_0z_1z_2+z_1^3=0.
$$
\end{example}

It is proved in \cite{Sha} that the quadric $Q_n$ admits a unique (up to isomorphism) generically transitive $\GG_a^n$-action corresponding to the pair $(A_n,U_n)$, where
$$
A_n=\KK[y_1,\dots,y_n]/(y_iy_j,\, y_i^2-y_j^2\, ;\, i\ne j) \ \ \text{and} \ \ U_n=\langle Y_1,\dots,Y_n\rangle
$$
for $n>1$, and $A_1=\KK[y_1]/(y_1^3)$, $U_1=\langle Y_1\rangle$.
Here by $Y_i$ we denote the image of the element $y_i$ in $A_n$. We describe below the $H$-pairs that determine generically transitive $\GG_a^n$-actions on degenerate quadrics. An algebra homomorphism $\phi:R\to R'$ is called a {\it homomorphism of $H$-pairs} $\phi:(R,U)\to (R',U')$ if $\phi(U)=U'$. Since the subspace $U'$ generates the algebra $R'$, a homomorphism of $H$-pairs is a surjective homomorphism of algebras.

\begin{proposition} \label{qpr1}
An $H$-pair $(R,U)$ determines a generically transitive $\GG_a^n$-action on the quadric $Q(n,k)$ if and only if there exists a homomorphism of $H$-pairs $\phi:(R,U)\to (A_k,U_k)$.
\end{proposition}

\begin{proof}
Suppose that a pair $(R,U)$ determines a generically transitive $\GG_a^n$-action on $Q(n,k)$. Let $q$ be the corresponding quadratic $\GG_a^n$-invariant form and $B$ the associated bilinear symmetric form. The operator of multiplication by an element $a\in U$ is skew-symmetric with respect to the form $B$, i.e., $B(ab,c)+B(b,ac)=0$ for every $b,c\in R$. If an element $c$ is in the kernel $I$ of the form $B$, then the first term is zero, and $ac\in I$. The subspace $U$ generates the algebra $R$, so $I$ is an ideal in $R$.

Let us show that $I\subseteq U$. Otherwise one has $\mm=U+I.$ Putting $b=c=1$ in the condition of skew-symmetry, we get $B(1,U)=0$ and  hence $B(1,\mm)=0$. For every $u_1,u_2\in U$ one has $B(u_1,u_2)+B(1,u_1u_2)=0$. It means that $B$ restricted to $U$ is zero. Note that $B(1,1)=0$, since the unit corresponds to a point on the quadric. But the form $B$ should be of rank at least 3, a contradiction.

The form $q$ induces a non-degenerate quadratic form $\overline{q}$ on the factor algebra $R/I$. The condition of skew-symmetry shows that $\overline{q}$ is $\GG_a^k$-invariant, where the $\GG_a^k$-action is given by the pair $(R/I,U/I)$. By \cite[Theorem~3]{Sha} this pair is isomorphic to $(A_k,U_k)$, and the homomorphism we need is the projection $(R,U)\to (R/I,U/I)$.

\smallskip

Conversely, assume that $\phi:(R,U)\to (A_k,U_k)$ is a homomorphism of $H$-pairs. One can lift the non-degenerate $\GG_a^k$-invariant form $\overline{q}$ on $A_k$ to a $\GG_a^n$-invariant form $q$ on $R$ by putting the kernel $I$ of the form $q$
equals $\Ker(\phi).$ We know that the $H$-pair $(R,U)$ determines a generically transitive $\GG_a^n$-action on some hypersurface and that $q(1)=0$; so this hypersurface is the quadric $q=0$.
\end{proof}

\begin{corollary} \label{qcor}
If $\phi:(R,U)\to(R',U')$ is a homomorphism of $H$-pairs and the $H$-pair $(R',U')$ is quadratic, then the $H$-pair $(R,U)$ is quadratic as well.
\end{corollary}

\begin{theorem} \label{qth}
An $H$-pair $(R,U)$ is quadratic if and only if $\mm^3\subset U$.
\end{theorem}

\begin{proof}
If $\mm^3\subset U$, then $(R/\mm^3, U/\mm^3)$ is an $H$-pair. If we prove that any pair $(R,U)$ with $\mm^3=0$ is quadratic, then Theorem~\ref{qth} follows from Corollary~\ref{qcor}. The algebra $R$ can be decomposed into the direct sum of its subspaces:  $R=\langle 1\rangle\oplus U\oplus\langle\mu\rangle$, where $\mu\in\mm^2$. Define a bilinear symmetric form $B$ on the space $R$ by $B(1,1)=0$, $B(1,\mu)=-1$, $B(1,U)=0$, and for every $x,y\in\mm$ obtain the value $B(x,y)$ from the equality $xy=u+B(x,y)\mu$, where $u\in U$. In this case the rank of the form $B$ is at least 3. Let us check that the operator of multiplication by an element $u\in U$ is skew-symmetric with respect to $B$. Fix a basis $e_0=1$, $e_1,\dots,e_n\in U$ and $e_{n+1}=\mu$ in the algebra $R$. It suffices to show that $B(ue_i,e_j)+B(e_i,ue_j)=0$ for all $i\le j$. Here both terms are zeroes apart from the case $i=0$, $j=1,\dots,n$, and $B(ue_0,e_j)$ is the coefficient at $\mu$ in the decomposition of the
 element $ue_j$, while $B(e_0,ue_j)$ is opposite to the coefficient at $\mu$ for the same element.

Conversely, suppose that an $H$-pair $(R,U)$ is quadratic. Then for some $k$ there exists a homomorphism
of $H$-pairs $\phi:(R,U)\to(A_k,U_k)$. If $\mm^3$ is not contained in $U$, then its image $\phi(\mm^3)$ is not contained in $U_k$, hence $\phi(\mm^3)\ne 0$. But in $A_k$ the cube of the maximal ideal is zero, a contradiction.
\end{proof}

\begin{corollary}
Consider a homomorphism $\phi:(R,U)\to(R',U')$ of $H$-pairs. The $H$-pair $(R,U)$ is quadratic if and only if the $H$-pair $(R',U')$ is quadratic.
\end{corollary}

Let us take a closer look at pairs $(R,U)$ determining generically transitive $\GG_a^n$-actions on the quadric $Q(n,n-1)$. There is a homomorphism $\phi:(R,U)\to(A_{n-1},U_{n-1})$ described in Proposition ~\ref{qpr1}. Its kernel $I$ should be one-dimensional, we have $\mm I=0$ and $\Ann(\mm)=\langle I,C\rangle$,
where $C\in \mm$ is an element with $\phi(C)=Y_1^2$. The multiplication determines a bilinear map $(U/I)\times(U/I)\to\Ann(\mm)$. Let us fix a basis vector $P$ in $I$ and obtain two bilinear symmetric forms $B_1$ and $B_2$ on $U/I$:
$$
xy=B_1(x+I,y+I)P+B_2(x+I,y+I)C \ \ \text{for} \ \text{all} \ x,y\in U.
$$
The form $B_2$ is non-degenerate. Not the vectors $P$ and $C$, but the line $\langle B_2\rangle$ and the linear span $\langle B_1,B_2\rangle$ are defined uniquely. If $B_1$ and $B_2$ are proportional, one can obtain $B_1=0$ by changing the basis; in this case we get the pair
$$
R=\KK[x_1,\dots,x_{n-1},p\,]/(x_ix_j, x_i^2-x_j^2, x_ip, p^2 ; i\ne j), \ \ U=\langle X_1,\dots,X_{n-1},P\rangle.
$$
Otherwise the flag $\langle B_2\rangle\subset\langle B_1,B_2\rangle$ in the space of bilinear forms on $U/I$ is well-defined. In an appropriate basis the form $B_2$ is represented by the identity matrix, and the matrix of the form $B_1$ is diagonal. One can assume that the second matrix has 1 and 0 as first two diagonal elements and other $n-3$ diagonal elements determine the parameters for isomorphism classes of the pairs. Hence there is an infinite family of pairwise non-equivalent generically transitive $\GG_a^4$-actions on the quadric $Q(4,3)$. Set the quadric $Q(4,3)$ by the equation  $z_1^2+z_2^2+z_3^2=2z_0z_5$ in $\PP^5$; the action of the element $(a_1,a_2,a_3,a_4)$ at the point $[z_0:z_2:z_2:z_3:z_4:z_5]$ is given by the formula
$$
[z_0:a_1z_0+z_1:a_2z_0+z_2:a_3z_0+z_3:\frac{a_2^2+a_3^2t+2a_4}{2}z_0+a_2z_2+ta_3z_3+z_4:
$$
$$
:\frac{a_1^2+a_2^2+a_3^2}{2}z_0+a_1z_1+a_2z_2+a_3z_3+z_5]
$$
depending on the parameter $t\in\KK$. If the forms $B_1$ and $B_2$ are assigned to the matrices $\text{diag}(1,0,t)$ and $E$, where $t\ne 0,1$, then the flag $\langle B_2\rangle\subset\langle B_1,B_2\rangle$ determines the value $t$ up to transformations  $t\to \frac{1}{t}$ and $t\to 1-t$. It means that different values $t_1$ and $t_2$ of the parameter determine equivalent actions if and only if
$$
\frac{(t_1^2-t_1+1)^3}{t_1^2(1-t_1)^2} \, = \, \frac{(t_2^2-t_2+1)^3}{t_2^2(1-t_2)^2}.
$$
The values $t=0$ and $t=1$ determine one more class of equivalent actions.


\section{Degree of a hypersurface}
\label{sec4}

The following result generalizes Theorem~\ref{qth}.

\begin{theorem} \label{dth}
Let $X$ be the closure of a generic orbit of
an effective $\GG_a^n$-action on $\PP^{n+1}$ corresponding to an $H$-pair $(R,U)$.
Then the degree of the hypersurface $X$ equals the maximal exponent $d$ such that
the subspace $U$ does not contain the ideal $\mm^d$.
\end{theorem}

\begin{proof}
Let $N$ be the maximal exponent with $\mm^N\ne 0$ and $d$ be the maximal exponent with $\mm^d\not\subset U$. Consider the flag $U_N\subseteq U_{N-1}\subseteq\dots\subseteq U_1=U$, where $U_i=\mm^i\cap U$, and construct a basis $\{e_i\}$ in $U$ coordinated with this flag. Add a vector $e\in\mm^d \setminus U$ to obtain a basis in $\mm$. For a nonzero vector $v\in\mm$ the maximal number $l$ such that $v\in\mm^l$ is called the {\it weight} of the vector and is denoted by $\omega(v)$. For example, $\omega(e)=d$. One may assume that the weights of the vectors $e_1,\dots,e_n$ do not decrease, that the vectors $e_1,\dots,e_s$ are the basis vectors of weight 1, and that the vectors $e_1,\dots,e_k$ are the basis vectors of weight $<d$. The hypersurface $X$ coincides with the closure of the projectivization of the set
$$
\{\exp(a_1e_1+\dots+a_ne_n)\, :\, a_1,\dots,a_n\in\KK\}.
$$
Suppose that
$\exp(a_1e_1+\dots+a_ne_n)=1+z_1e_1+\dots+z_ne_n+ze$.
Then $z_i=a_i+f_i(a_1,\dots,a_{r_i})$, where the weights of $e_1,\dots,e_{r_i}$ do not exceed $\omega(e_i)$. In particular, $z_1=a_1,\dots,z_s=a_s$. Assume that $i>s$. For any term $\alpha a_1^{j_1}\dots a_{r_i}^{j_{r_i}}$, $\alpha\in\KK^{\times}$, of the polynomial $f_i$ one has $j_1\omega(e_1)+\dots+j_{r_i}\omega(e_{r_i})\le\omega(e_i)$.

Observe that $a_i=z_i-f_i(a_1,\dots,a_{r_i})$. One can show by induction that any element $a_p$, $p=1,\dots,r_i$, can be expressed as a polynomial $F_p$ of degree $\le\omega(e_p)$ in $z_1,\dots,z_{r_p},z_p$. It means that $a_i$ can be expressed as a polynomial $F_i$ of degree $\le\omega(e_i)$ in $z_1,\dots,z_{r_i},z_i$. Finally, we have $z=f(a_1,\dots,a_k)$, where for every term $\alpha a_1^{j_1}\dots a_k^{j_k}$, $\alpha\in\KK^{\times}$, one has $j_1\omega(e_1)+\dots+j_k\omega(e_k)\le d$. Replace each
$a_i$, $i=1,\dots,k$ by its expression in $z_1,\dots,z_{r_i},z_i$ and obtain $z=F(z_1,\dots,z_k)$. Therefore the degree of every term of the polynomial $F$ does not exceed $d$.

For any polynomial $G(y_1,\dots,y_m)$ of degree $r$ one can multiply each term of degree $q$ by $z_0^{r-q}$ to get a homogeneous polynomial that we denote by $HG(z_0,y_1,\dots,y_m)$.  The polynomial $\widetilde{F}(z_1,\dots,z_k,z)=z-F(z_1,\dots,z_k)$ being linear in $z$ is irreducible.
Thus the closure of the orbit $\GG_a^n\langle 1\rangle$ is given by the equation $H\widetilde{F}(z_0,z_1,\dots,z_k,z)=0$, where $z_0$ denotes the first coordinate in the basis $\{1,e_1,\dots,e_n,e\}$ of the algebra $R$.

We have to show that there is a term of degree $d$ in the polynomial $F$. We use induction on $n$. If $n=1$, then the corresponding pair is $(R,U)=(\KK[y]/(y^3), \langle y\rangle)$. Here $d=2$, and the pair is quadratic by Theorem ~\ref{qth}. Suppose $n \ge 2$. Then the linear span $I=\langle e_{k+1},\dots,e_n\rangle$ is a subset of $U$ and an ideal in $R$. As we know, the polynomials $F(z_1,\dots,z_k)$ for two pairs $(R,U)$ and $(R/I, U/I)$ coincide. So the assumption is applicable if $I\ne 0$. There is only one case left, namely, $I=0$, or, equivalently, $k=n$, $\mm^d=\langle e\rangle$ and $\mm^{d+1}=0$. In this situation the element $e$ can be represented as a homogeneous form of degree $d$ in $e_1,\dots,e_s$. It is well known that every homogeneous form of degree $d$ can be decomposed as the sum of $d$-th powers of linear forms. Hence $e$ is equal to a $d$-th power of some linear combination of $e_1,\dots,e_s$. Changing the basis, one can assume that $e_1^d=e$ and that th
 e vectors $e_1,e_1^2,\dots,e_1^{d-1}$ form the subset $\{e_i\, :\, i\in \mathcal{I}\}$ of the basis $\{e_1,\dots,e_n\}$ of the space $U$. Then for any $j\notin\mathcal{I}$ the polynomial $f_j$ has no terms depending only on $a_1$. Further, if $i\in\mathcal{I}$, then $f_i$ contains the term $\frac{1}{\omega(e_i)!}a_1^{\omega(e_i)}$. One can show by induction that the polynomial $F_i$ contains some term depending on $z_i$ whenever $i\in\mathcal{I}$. So we can presume that the pair $(R,U)$ is of the form $(\KK[y]/(y^{d+1}),\langle y,y^2,\dots,y^{d-1}\rangle)$ for calculating the coefficient at $z_1^d$ in the polynomial $F(z_1,\dots,z_k)$. For this pair we have $e_1=y, e_2=y^2,\dots,e_{d-1}=y^{d-1}$ and $e=y^d$. Thus
$$
\exp(a_1e_1+\dots+a_{d-1}e_{d-1})=1+z_1e_1+\dots +z_{d-1}e_{d-1}+ze
$$
and
$$
z=\sum_{m>0}\frac{1}{m!}\sum_{\substack{i_1k_1+\dots+i_mk_m=d \\ k_1<\dots<k_m<d}} \frac{m!}{i_1!\dots i_m!}a_{k_1}^{i_1}\dots a_{k_m}^{i_m} \ = \
$$
\begin{equation} \label{for}
\ =\ \sum_{\substack{i_1k_1+\dots+i_mk_m=d\\ k_1<\dots<k_m<d}} \frac{1}{i_1!\dots i_m!}a_{k_1}^{i_1}\dots a_{k_m}^{i_m}.
\end{equation}
Recall that $a_1=z_1$ and $a_i=z_i-f_i(a_1,\dots,a_{i-1})$ as $i>1$. Here the polynomial $f_i$ is given by formula ~(\ref{for}), if we replace $d$ by $i$. Now substitute all the $a_j$'s in the polynomials $f_i$ by their expressions in $z_1,\dots,z_j$ to obtain the polynomial $F_i(z_1,\dots,z_{i-1})$. Denote by $c_i$ the coefficient at $z_1^i$ in the expression $z_i-F_i(z_1,\dots,z_{i-1})$. Then $c_1=1$ and
$$
c_p=-\sum_{\substack{i_1k_1+\dots+i_mk_m=p\\ k_1<\dots<k_m<p}} \frac{1}{i_1!\dots i_m!}c_{k_1}^{i_1}\dots c_{k_m}^{i_m}
$$
as $p>1$. These coefficients are given by the same recurrence relation as the coefficients of the series
$$
h(t)=t-\sum_{u\ge 2}\frac{1}{u!}(t-\frac{t^2}{2}+\frac{t^3}{3}-\frac{t^4}{4}+\dots)^u=c_1t+c_2t^2+c_3t^3+\dots.
$$
This is the series $t-\exp(\ln(1+t))+1+\ln(1+t)=\ln(1+t)$. It means that $c_p=\frac{(-1)^{p-1}}{p}$. So the coefficient at $z_1^d$ in the polynomial $F(z_1,\dots,z_k)$ is nonzero. This concludes the proof of Theorem~\ref{dth}.
\end{proof}

\begin{corollary}
Let $X$ be the closure of a generic orbit of
an effective $\GG_a^n$-action on $\PP^{n+1}$. Then
$\deg X \le n+1$.
\end{corollary}

\section{Concluding remarks}

 We would like to mention several classes of algebraic varieties where the classification of generically transitive
$\GG_a^n$-actions may be valuable.

 Let $G$ be a connected semisimple algebraic group and $P$ be a parabolic subgroup of $G$. The homogeneous space $G/P$ is a projective variety called a (generalized) flag variety. All flag varieties admitting a generically transitive $\GG_a^n$-action are found in~\cite{Ar}. Recently E.~Feigin~\cite{Fe} has constructed a $\GG_a^n$-degeneration for arbitrary
flag varieties. On the other side, the problem of description of all generically transitive $\GG_a^n$-actions on a given flag variety is still open. Hassett-Tschinkel correspondence provides such a description for $\PP^n\cong \SL(n+1)/P_1$, where
$P_1$ is a maximal parabolic subgroup in $\SL(n+1)$ corresponding to the first simple root. By \cite{Sha},
a generically transitive $\GG_a^n$-action on the non-degenerate quadric $Q_n\cong \SO(n+2)/P_1$ is unique.
We expect further uniqueness results for the Grassmannians $\text{Gr}(k,n)$ with $2\le k\le n-2$.

 The study of generically transitive $\GG_a^n$-actions on $G$-varieties is related to classification of maximal commutative unipotent subgroups in $G$. More precisely, let $G$ be an affine algebraic group acting on a variety $X$ and $F$ be a commutative unipotent subgroup of $G$. Assume that the restricted $F$-action on $X$ is generically transitive. Then $F$ is a maximal commutative unipotent subgroup of $G$. For $G=\text{SL(n+1)}$, a characterization of maximal commutative unipotent subgroups which act generically transitively on $\PP^n$ may be obtained from~\cite[Theorem~15]{ST}.

 The classification of singular del Pezzo surfaces with a generically transitive $\GG_a^2$-action is obtained
by U.~Derenthal and D.~Loughran \cite{DL}.

 Finally let us return to toric varieties. Applying the blow-up construction, one obtains many
generically transitive $\GG_a^n$-actions on toric varieties starting from $\PP^n$. But how to get a complete
classification of $\GG_a^n$-actions? The first instructive examples are the Hirzebruch surfaces $\FF_n$. It is proved in
\cite[Proposition~5.5]{HT} that for $n>0$ the surface $\FF_n$ carries two distinct generically transitive
$\GG_a^2$-actions.

\section*{Acknowledgments}

The authors are grateful to the referee for a careful reading and many helpful suggestions and remarks.


\end{document}